\documentclass{amsart}
\usepackage{amsthm, amssymb, amsfonts, amscd}
\usepackage{graphicx}
\usepackage{tikz-cd}
\usepackage{verbatim}
\usepackage{enumitem}
\usepackage{mathrsfs, mathtools}
\usepackage{xcolor}
\usepackage[colorlinks = true,
linkcolor = black,
urlcolor  = blue,
citecolor = black]{hyperref}
\usepackage{cleveref}
\usepackage[margin=1in]{geometry}
\usepackage{setspace}
\usepackage{diagbox}
\allowdisplaybreaks

\theoremstyle{definition} 
\newtheorem{thm}{Theorem}[section]

\newtheorem{prop}[thm]{Proposition}
\newtheorem{lem}[thm]{Lemma}

\newcommand{\mr}{\mathrm}
\newcommand{\iu}{\mathrm{i}}

\newcommand{\mc}{\mathcal}

\newcommand{\Aut}{\mr{Aut}}

\newcommand{\bb}{\mathbb}

\newcommand{\bZ}{\bb{Z}}

\newcommand{\bR}{\bb{R}}
\newcommand{\bC}{\bb{C}}
\newcommand{\bP}{\bb{P}}

\newcommand{\bF}{\bb{F}}
\newcommand{\bE}{\bb{E}}
\newcommand{\fp}{\bF_p}

\newcommand{\zp}{\bZ_p}

\newcommand{\ol}{\overline}
\newcommand{\Sur}{\mr{Sur}}
\newcommand{\M}{\mr{M}}

\newcommand{\ze}{\zeta}

\newcommand{\al}{\alpha}

\newcommand{\ga}{\gamma}

\newcommand{\cok}{\mr{cok}}
\newcommand{\Hom}{\mr{Hom}}
\newcommand{\rank}{\mr{rank}}
\newcommand{\ra}{\rightarrow}

\newcommand{\lt}{\left |}
\newcommand{\rt}{\right |}
\newcommand{\lf}{\left \lfloor}
\newcommand{\rf}{\right \rfloor}
\begin{document}

\title[Sharp threshold for random matrices over finite fields]{Sharp threshold for universality of cokernels of random matrices over finite fields}
\author{Jungin Lee}
\date{}
\address{J. Lee -- Department of Mathematics, Ajou University, Suwon 16499, Republic of Korea}
\email{jileemath@ajou.ac.kr}

\begin{abstract}
In this paper, we determine the sharp threshold for universality of cokernels of random matrices over finite fields. 
More precisely, we prove the following: given any constant $c>1$, let $(A(n))_{n \ge 1}$ be a sequence of random $n \times n$ matrices over $\mathbb{F}_p$ such that, for all sufficiently large $n$, the entries of $A(n)$ are independent and take any given value of $\mathbb{F}_p$ with probability at most $1 - \frac{c \log n}{n}$. 
Then the cokernels of $A(n)$ converge in distribution, as $n \to \infty$, to the same limiting law as the cokernels of uniform random $n \times n$ matrices over $\mathbb{F}_p$. This answers an open problem posed by Wood.
\end{abstract}
\maketitle

\vspace{-5mm}
\section{Introduction} \label{Sec1}

For a prime $p$, let $\fp$ denote the finite field with $p$ elements and $\zp$ denote the ring of $p$-adic integers.
For a commutative ring $R$, write $\M_n(R)$ for the set of all $n \times n$ matrices over $R$. For $A \in \M_n(R)$, regard $A$ as an $R$-linear map $R^n\to R^n$ ($x \mapsto Ax$) and define the \emph{cokernel} of $A$ by $\cok(A) := R^n/\mr{im}(A)$. When $R=\fp$, we have $\dim_{\fp}\cok(A)=n-\rank(A)$, so the isomorphism class of $\cok(A)$ is determined by the corank of $A$. 
Throughout the paper, Haar measure means the Haar probability measure. On a finite group, this is the uniform measure.

Friedman and Washington \cite{FW89} proved that if $A(n)$ is a random matrix in $\M_n(\zp)$ distributed according to the Haar measure for each $n \ge 1$, then
\begin{equation*}
\lim_{n \to \infty} \bP(\cok(A(n)) \cong H) = \frac{1}{|\Aut(H)|} \prod_{i=1}^{\infty} (1-p^{-i})
\end{equation*}
for every finite abelian $p$-group $H$, where $\Aut(H)$ denotes the automorphism group of $H$.

Wood \cite{Woo19} and Nguyen and Wood \cite{NW22} showed that the same limiting distribution holds for much broader classes of random matrices $A(n)$ whose entries are independent, and for which the reduction modulo $p$ of each entry is not too concentrated on any single value of $\fp$. 
To state their result, we recall the following definition. Let $\al$ be a positive real number. 
A random element $x$ in the finite field $\fp$ is $\al$-\textit{balanced} if 
$$
\bP(x = r) \leq 1 - \al
$$ 
for every $r \in \fp$. A random element $x$ in $\zp$ is $\al$-\textit{balanced} if its reduction modulo $p$ is $\al$-\textit{balanced} as a random element in $\fp$. A random matrix $A$ in $\M_n(R)$ (where $R = \fp$ or $\zp$) is $\al$-\textit{balanced} if its entries are independent and $\al$-balanced. We note that in the definition of an $\al$-balanced matrix, the entries are not required to be identically distributed.

\begin{thm} \label{thm1a}
(\cite[Theorem 4.1]{NW22}) Let $(\al_n)_{n \ge 1}$ be a sequence of positive real numbers such that for every constant $\Delta > 0$, we have $\al_n \geq \frac{\Delta \log n}{n}$ for all sufficiently large $n$. Let $A(n)$ be an $\al_n$-balanced random matrix in $\M_n(\zp)$ for each $n \ge 1$. Then for every finite abelian $p$-group $H$,
\begin{equation*}
\lim_{n \to \infty} \bP(\cok(A(n)) \cong H) = \frac{1}{|\Aut(H)|} \prod_{i=1}^{\infty} (1-p^{-i}).
\end{equation*}
\end{thm}

By reduction modulo $p$, one obtains the following corollary for random matrices over $\fp$.

\begin{thm} \label{thm1b}
Let $(\al_n)_{n \ge 1}$ be a sequence of positive real numbers such that for every constant $\Delta > 0$, we have $\al_n \geq \frac{\Delta \log n}{n}$ for all sufficiently large $n$. Let $A(n)$ be an $\al_n$-balanced random matrix in $\M_n(\fp)$ for each $n \ge 1$. Then for every nonnegative integer $k$,
\begin{equation*}
\lim_{n \to \infty} \bP(\cok(A(n)) \cong \fp^k)
= \lim_{n \to \infty} \bP(\rank(A(n)) = n-k)
= p^{-k^2} \frac{\prod_{j=1}^{\infty} (1-p^{-j})}{\prod_{j=1}^{k} (1-p^{-j})^2}.
\end{equation*}
\end{thm}

In the above results, the distribution of $\cok(A(n))$ converges to the same law, independent of how the entries of $A(n)$ are distributed. These universality results have been extended to the distribution of the cokernels of various types of random $p$-adic matrices: Wood \cite{Woo17} proved the result for symmetric matrices, Nguyen and Wood \cite{NW25} for alternating matrices, and the author \cite{Lee23} for Hermitian matrices.
Further universality results have also been obtained, including for the joint distributions of multiple cokernels \cite{CY23, HNVP25, Lee24, NVP24}, as well as for random matrices over finite fields \cite{LMN21, Map10, She25}. See also \cite{Hod24, KLY26, Yan25} for various universality results concerning cokernels of random $p$-adic matrices.

The rank distribution of sparse i.i.d. random matrices over finite fields was studied earlier by Balakin \cite{Bal68}, Blömer, Karp and Welzl \cite{BKW97}, and Cooper \cite{Coo00}. They obtained finer rank and nonsingularity estimates near the scale $q_n\asymp(\log n)/n$ for the more restrictive model in which an entry is zero with probability $1-q_n$ and, conditional on being nonzero, is uniform on $\fp^*$. By contrast, Theorem \ref{thm1_main} below allows each entry to have a distribution depending on both $n$ and its position, without requiring the entries to be identically distributed or to follow this particular law.

It is natural to ask whether the lower bound on $\al_n$ in Theorems \ref{thm1a} and \ref{thm1b} can be improved. The scale $\frac{\log n}{n}$ is forced by zero columns, and a fixed additive improvement does not suffice under the present hypotheses. Indeed, fix $s \in \bR$. For all sufficiently large $n$, set
$q_n=\frac{\log n+s}{n} \in (0,1)$, and let $A_s(n)\in\M_n(\fp)$
be a random matrix whose entries are i.i.d. copies of a random element
$\xi_n\in\fp$ satisfying
$$
\bP(\xi_n=0)=1-q_n, \qquad \bP(\xi_n=1)=q_n.
$$
For the remaining finitely many values of $n$, define $A_s(n)$ arbitrarily. Then $A_s(n)$ is $q_n$-balanced for all sufficiently large $n$.
For all sufficiently large $n$, if $Z_n$ denotes the number of zero columns of $A_s(n)$, then 
$$
Z_n \sim \mr{Bin}\left(n,(1-q_n)^n\right),
$$ 
where $\mr{Bin}(N, \theta)$ denotes the binomial distribution with parameters $N$ and $\theta$. Then 
$$
(1-q_n)^n = \exp\left(n\log(1-q_n)\right) = \frac{e^{-s}}{n}(1+o(1))
$$
so $(1-q_n)^n\to 0$ and $n(1-q_n)^n\to e^{-s}$ as $n \to \infty$. By the standard Poisson limit theorem, the distribution of $Z_n$ converges to the Poisson distribution $\mr{Pois}(e^{-s})$ of mean $e^{-s}$.

Since $\dim_{\fp}\cok(A_s(n))\ge Z_n$, for every fixed $m\ge 0$ we have
$$
\liminf_{n\to\infty}
\bP(\dim_{\fp}\cok(A_s(n))\ge m)
\ge \bP(\mr{Pois}(e^{-s})\ge m)
\ge \bP(\mr{Pois}(e^{-s})=m)
=\frac{e^{-e^{-s}}e^{-sm}}{m!}.
$$
On the other hand, the tail of the uniform limiting law in Theorem \ref{thm1b} satisfies
$$
\sum_{k=m}^{\infty} p^{-k^2}
\frac{\prod_{j=1}^{\infty}(1-p^{-j})}
{\prod_{j=1}^{k}(1-p^{-j})^2}=O_p(p^{-m^2}),
$$
which is smaller than $\frac{e^{-e^{-s}}e^{-sm}}{m!}$ for all sufficiently large $m$. Hence the corank distributions of $A_s(n)$ cannot converge to the uniform limiting law in Theorem \ref{thm1b}. The finer question of whether $\frac{\log n+\omega(1)}{n}$ is sufficient is not addressed in this paper.

Based on the above observation, Wood \cite[Open Problem 3.3]{Woo23}, in her ICM lecture, posed the problem of lowering the bound on $\al_n$ in Theorems \ref{thm1a} and \ref{thm1b} as close as possible to $\frac{\log n}{n}$.
The best previously known bound required that, for every $\Delta > 0$, $\al_n \ge \frac{\Delta \log n}{n}$ for all sufficiently large $n$.
This leaves a substantial gap between the known universality regime and the critical scale $\frac{\log n}{n}$.

The aim of this paper is to close this gap in the finite field setting, thereby resolving \cite[Open Problem 3.3]{Woo23} in the finite field case.

\begin{thm} \label{thm1_main}
Let $c>1$ be a constant and set $\al_n = \frac{c \log n}{n}$. 
Let $(A(n))_{n\ge1}$ be a sequence of random matrices with $A(n)\in\M_n(\fp)$ for every $n\ge1$, and assume that $A(n)$ is $\al_n$‑balanced for all sufficiently large $n$.
Then for every nonnegative integer $k$, 
\begin{equation} \label{eq1_main1}
\lim_{n \to \infty} \bP(\cok(A(n)) \cong \fp^k) 
= \lim_{n \to \infty} \bP(\rank(A(n)) = n-k)
= p^{-k^2} \frac{\prod_{j=1}^{\infty} (1-p^{-j})}{\prod_{j=1}^{k} (1-p^{-j})^2}.
\end{equation}
\end{thm}

By \cite[Theorem 1]{FK06} (see also \cite[Theorem 2.3]{Woo23}), the limiting distribution of $\cok(A(n))$ is determined by its limiting surjective moments. Thus, Theorem \ref{thm1_main} follows from the next theorem. We write $\Hom(A,B)$ (resp. $\Sur(A,B)$) for the set of all homomorphisms (resp. surjective homomorphisms) from $A$ to $B$.

\begin{thm} \label{thm1_main2}
Let $c>1$ be a constant and set $\al_n = \frac{c \log n}{n}$. 
Let $(A(n))_{n\ge1}$ be a sequence of random matrices with $A(n)\in\M_n(\fp)$ for every $n\ge1$, and assume that $A(n)$ is $\al_n$‑balanced for all sufficiently large $n$. Then for every nonnegative integer $r$, 
\begin{equation} \label{eq1_main2}
\lim_{n \to \infty} \bE(\# \Sur(\cok(A(n)), \fp^r)) = 1.
\end{equation}
\end{thm}

For comparison, if $U(n)$ is uniform in $\M_n(\fp)$, then
\cite[Theorem 2.9]{Woo19} gives
$$
\lim_{n\to\infty} \bE(\#\Sur(\cok(U(n)),\fp^r))=1.
$$
Thus, the limit in \eqref{eq1_main2} is the surjective moment of the uniform random matrix model.

After the initial version of this paper, Jung, the author and Yu \cite{JLY26} developed a unified framework for proving sharp threshold universality for nonsymmetric, symmetric and alternating random matrix models over $\zp$. This paper treats the finite field nonsymmetric case that served as a starting point for this later framework: the sets of popular values and the decomposition introduced here provide the core Fourier-analytic argument near the sharp threshold, while \cite{JLY26} provides additional moment-theoretic and structural ingredients needed to establish the corresponding universality results for nonsymmetric, symmetric and alternating random matrices over $\zp$.

The paper is organized as follows.
In Section \ref{Sub21}, we reduce Theorem \ref{thm1_main2} to the special case in which each entry of $A(n)$ equals $0$ with probability $1-\al_n$ and equals a fixed element in $\fp^*$ (which may vary across entries) with probability $\al_n$.
Fix a positive integer $r$. For each sufficiently large $n$ and 
$t=(t_{kl})\in(\fp^*)^{[n]\times[n]}$, define
$$
F(t) = \frac{1}{p^{rn}}\sum_{i=r}^{n} \sum_{j=1}^{n} \sum_{\substack{1 \le k_1 < \cdots < k_i \le n \\ 1 \le l_1 < \cdots < l_j \le n}} \sum_{\substack{g_{k_1}, \ldots, g_{k_i} \in \fp^r \setminus \{0 \} \\ \left< g_{k_1}, \ldots, g_{k_i} \right> = \fp^r}}
\sum_{h_{l_1}, \ldots, h_{l_j} \in \fp^r \setminus \{0 \}} \left | \prod_{a=1}^{i} \prod_{b=1}^{j} (1-\al_n + \al_n \ze^{t_{k_a l_b}(g_{k_a} \cdot h_{l_b})}) \right |,
$$
where $\ze=\exp(2\pi\iu/p)$. Theorem \ref{thm1_main2} is then
reduced to proving the uniform estimate
$$
\lim_{n \to \infty}\sup_{t \in (\fp^*)^{[n]\times[n]}}F(t)=0.
$$
We note that all estimates in the proof are uniform over all arrays $t=(t_{kl})$.

To estimate the complicated sum $F(t)$, we decompose it into partial sums according to the support sizes $i$ and $j$. Let $I_1$, $I_2$ and $I_3$ be the sets of ``small'', ``intermediate'' and ``large'' elements of $[n]$, respectively. (See the paragraph before Theorems \ref{thm21_main2a}, \ref{thm21_main2b} and \ref{thm21_main2c} for the precise definitions.) When $i \in I_1$ or $j \in I_1$, factorization of the Fourier sum yields a power saving in the corresponding support size, which dominates the number of possible supports and nonzero values. When $i \in I_2$ or $j \in I_2$, we obtain a uniform bound for each summand, while an entropy estimate controls the number of such summands. Combining these estimates shows that the total contribution tends to zero as $n \to \infty$.

The main difficulty lies in estimating the contribution from the range $i,j \in I_3$ (Theorem \ref{thm21_main2c}). Even in this range, the quantity
$$
N := \left | \{ (a,b) : 1 \le a \le i, 1 \le b \le j \text{ and } g_{k_a} \cdot h_{l_b} \ne 0 \} \right |
$$
may not be large enough.

We overcome this issue by introducing the sets $A = A_{g_{k_1}, \ldots, g_{k_i}}$ and $B = B_{h_{l_1}, \ldots, h_{l_j}}$ consisting of the popular values among $\{ g_{k_a} \}_{a=1}^{i}$ and $\{ h_{l_b} \}_{b=1}^{j}$, respectively (see Section \ref{Sub23} for the precise definitions).
According to the structure of $A$ and $B$, we divide the case in which both $i$ and $j$ are large into three subcases:
\begin{enumerate}
\item[($\mc{C}_1$)] $g \cdot h \ne 0$ for some $(g,h) \in A \times B$;
\item[($\mc{C}_2$)] $A \subseteq V \setminus \{0 \}$ and $B \subseteq V^{\perp} \setminus \{0\}$ for some $\fp$-subspace $V$ of $\fp^r$, and at least one inclusion is strict;
\item[($\mc{C}_3$)] $A = V \setminus \{0 \}$ and $B = V^{\perp} \setminus \{0\}$ for some $\fp$-subspace $V$ of $\fp^r$.
\end{enumerate}

In the case $\mc{C}_1$, there exist $g\in A$ and $h\in B$ with $g\cdot h\ne 0$. Since both $g$ and $h$ occur at least $\ga_3n$ times (for some constant $\ga_3>0$), this gives a sufficiently strong bound for each corresponding summand. In the case $\mc{C}_2$, elements of $A$ are orthogonal to elements of $B$, but at least one of $A$ and $B$ is a proper subset of the corresponding nonzero subspace. This allows us to bound the number of possible tuples sufficiently strongly. The case $\mc{C}_3$ requires a further analysis of the values outside $V$ and $V^\perp$. We prove the estimates for $\mc{C}_1$ and $\mc{C}_2$ in Section \ref{Sub23}. In Section \ref{Sub24}, we subdivide $\mc{C}_3$ into three subcases $\mc{C}_{3,a}$, $\mc{C}_{3,b}$ and $\mc{C}_{3,c}$, and prove that the contribution of each subcase converges to $0$ as $n\to\infty$.

\section{Proof of main theorem} \label{Sec2}

We begin by introducing some notation.
For a positive integer $n$, denote $[n] := \{ 1, 2, \ldots, n \}$.
For any set $S$, write $S^{[n] \times [n]} := \{ (s_{kl})_{k,l \in [n]} : s_{kl} \in S \}$. 
Since Theorem \ref{thm1_main2} is trivial when $r=0$, it suffices to consider the case when $r > 0$. Fix a positive integer $r$ and write $G := \fp^r$ and $G^* := G \setminus \{ 0 \}$. For $g_1, \ldots, g_k \in G$, let $\left< g_1, \ldots, g_k \right>$ be the $\fp$-subspace of $G$ generated by $g_1, \ldots, g_k$.

Let $c>1$ be a constant and set $\al_n = \frac{c \log n}{n}$. Let $(A(n))_{n\ge1}$ be a sequence of random matrices with $A(n)\in\M_n(\fp)$ for every $n\ge1$, and assume that $A(n)$ is $\al_n$‑balanced for all sufficiently large $n$. 
Since all assertions are asymptotic, we henceforth restrict to sufficiently large $n$ for which this condition holds and $\al_n<\frac{1}{2}$. We denote by $A(n)_l$ the $l$-th column of $A(n)$ and by $A(n)_{kl}$ the $(k,l)$ entry of $A(n)$. We write $\bP$ for probability and $\bE$ for expected value. Finally, let $\ze := \exp(2 \pi \iu/p) \in \bC$.

\subsection{Reduction to extreme points} \label{Sub21}

We regard $A(n)$ as a linear map from $\fp^n$ to $\fp^n$. Then a surjective homomorphism from $\cok(A(n))$ to $G$ corresponds to an element $F\in\Sur(\fp^n,G)$ satisfying $FA(n)=0$. By a direct computation as in \cite[Section 2]{Woo19}, we obtain
\begin{align*}
\bE(\# \Sur(\cok(A(n)), G)) 
& = \sum_{F \in \Sur(\fp^n, G)} \bP(FA(n)=0) \\
& = \sum_{F \in \Sur(\fp^n, G)} \prod_{l=1}^{n} \bP(FA(n)_l=0) \\
& = \sum_{F \in \Sur(\fp^n, G)} \prod_{l=1}^{n} \frac{1}{|G|}\sum_{C_l \in \Hom(G, \fp)} \bE(\ze^{C_l(FA(n)_l)}) \\
& = \frac{1}{p^{rn}} \sum_{F \in \Sur(\fp^n, G)} \prod_{l=1}^{n} \sum_{h_l \in G} \bE(\ze^{FA(n)_l \cdot h_l}) \\
& = \frac{1}{p^{rn}} \sum_{\substack{g_1, \ldots, g_n \in G \\ \left< g_1, \ldots, g_n \right> = G}} \prod_{l=1}^{n} \sum_{h_l \in G} \prod_{k=1}^{n} \bE(\ze^{A(n)_{kl} (g_k \cdot h_l)}) \\
& = \frac{1}{p^{rn}} \sum_{\substack{g_1, \ldots, g_n \in G \\ \left< g_1, \ldots, g_n \right> = G}} \sum_{h_1, \ldots, h_n \in G} \prod_{k=1}^{n} \prod_{l=1}^{n} \bE(\ze^{A(n)_{kl} (g_k \cdot h_l)}).
\end{align*}
In the above equation, the third equality uses the discrete Fourier transform on $G$, the fourth follows by writing each $C_l \in \Hom(G, \fp)$ as $C_l(g)=g \cdot h_l$ for some $h_l \in G$, and the fifth follows from writing $F(e_k)=g_k$, where $e_1, \ldots, e_n$ form the standard basis of $\fp^n$, and using the independence of the entries of $A(n)$.

Since the contribution of the tuple $(h_1, \ldots, h_n)=(0, \ldots, 0)$ is
$$
\frac{1}{p^{rn}} \sum_{\substack{g_1, \ldots, g_n \in G \\ \left< g_1, \ldots, g_n \right> = G}} 1
= \frac{1}{p^{rn}} \prod_{j=0}^{r-1} (p^n-p^j)
= \prod_{j=0}^{r-1} \left ( 1 - \frac{1}{p^{n-j}} \right ),
$$
which converges to $1$ as $n \to \infty$, it follows that \eqref{eq1_main2} holds if and only if
\begin{align*}
& \frac{1}{p^{rn}} \sum_{\substack{g_1, \ldots, g_n \in G \\ \left< g_1, \ldots, g_n \right> = G}} \sum_{\substack{h_1, \ldots, h_n \in G \\ (h_1, \ldots, h_n) \ne (0, \ldots, 0) }} \prod_{k=1}^{n} \prod_{l=1}^{n} \bE(\ze^{A(n)_{kl} (g_k \cdot h_l)}) \\
= \, & \frac{1}{p^{rn}} \sum_{i=r}^{n} \sum_{j=1}^{n} \sum_{\substack{1 \le k_1 < \cdots < k_i \le n \\ 1 \le l_1 < \cdots < l_j \le n}} \sum_{\substack{g_{k_1}, \ldots, g_{k_i} \in G^* \\ \left< g_{k_1}, \ldots, g_{k_i} \right> = G}}
\sum_{h_{l_1}, \ldots, h_{l_j} \in G^*} \prod_{a=1}^{i} \prod_{b=1}^{j} \bE(\ze^{A(n)_{k_a l_b} (g_{k_a} \cdot h_{l_b})})
\end{align*}
converges to $0$ as $n \to \infty$. The equality holds since $\bE(\ze^{A(n)_{kl} (g_k \cdot h_l)})=1$ whenever $g_k=0$ or $h_l=0$, and the condition $\left< g_1, \ldots, g_n \right> = G$ (resp. $(h_1, \ldots, h_n) \ne (0, \ldots, 0)$) implies that $i \ge r$ (resp. $j \ge 1$).

Let $\mr{Ran}(\fp,\al_n)$ be the set of all $\al_n$-balanced random elements in $\fp$, where random elements with the same distribution are identified. Then $A(n)_{kl}$ is an element of $\mr{Ran}(\fp, \al_n)$ for all $k,l \in [n]$. 
Define a function $f_{p,r} : \mr{Ran}(\fp, \al_n)^{[n] \times [n]} \to \bC$ by
\begin{align*}
f_{p,r}((x_{kl})) := \frac{1}{p^{rn}}\sum_{i=r}^{n} \sum_{j=1}^{n} \sum_{\substack{1 \le k_1 < \cdots < k_i \le n \\ 1 \le l_1 < \cdots < l_j \le n}} \sum_{\substack{g_{k_1}, \ldots, g_{k_i} \in G^* \\ \left< g_{k_1}, \ldots, g_{k_i} \right> = G}}
\sum_{h_{l_1}, \ldots, h_{l_j} \in G^*} \prod_{a=1}^{i} \prod_{b=1}^{j} \bE(\ze^{x_{k_a l_b} (g_{k_a} \cdot h_{l_b})})
\end{align*}
and a function $F_{p,r} : \mr{Ran}(\fp, \al_n)^{[n] \times [n]} \to \bR$ by
$$
F_{p,r}((x_{kl})) := \frac{1}{p^{rn}}\sum_{i=r}^{n} \sum_{j=1}^{n} \sum_{\substack{1 \le k_1 < \cdots < k_i \le n \\ 1 \le l_1 < \cdots < l_j \le n}} \sum_{\substack{g_{k_1}, \ldots, g_{k_i} \in G^* \\ \left< g_{k_1}, \ldots, g_{k_i} \right> = G}}
\sum_{h_{l_1}, \ldots, h_{l_j} \in G^*} \left | \prod_{a=1}^{i} \prod_{b=1}^{j} \bE(\ze^{x_{k_a l_b} (g_{k_a} \cdot h_{l_b})}) \right |.
$$
As explained above, it suffices to prove the following uniform statement for every positive integer $r$:
\begin{equation} \label{eq21_func}
\lim_{n \to \infty}\sup_{(x_{kl}) \in \mr{Ran}(\fp, \al_n)^{[n] \times [n]}}
\left | f_{p,r}((x_{kl})) \right | = 0.
\end{equation}

For each $t = (t_{kl}) \in (\fp^*)^{[n] \times [n]}$, let $X(t) = (x(t)_{kl})$ be the element of $\mr{Ran}(\fp, \al_n)^{[n] \times [n]}$ defined by
$$
\bP(x(t)_{kl}=u) = \begin{cases}
1-\al_n & \text{if } u=0, \\
\al_n & \text{if } u=t_{kl}, \\
0 & \text{otherwise}
\end{cases}
$$
for every $k,l \in [n]$. The following lemma reduces Theorem \ref{thm1_main2} to random matrices $A(n)$ whose $(k,l)$ entry equals $0$ with probability $1-\al_n$ and equals a fixed value $t_{kl} \in \fp^*$ with probability $\al_n$.

\begin{lem} \label{lem2a}
For all sufficiently large $n$ and every $(x_{kl}) \in \mr{Ran}(\fp, \al_n)^{[n] \times [n]}$, there exists $t = (t_{kl}) \in (\fp^*)^{[n] \times [n]}$ such that 
$$
\lt f_{p,r}((x_{kl})) \rt \le F_{p,r}(X(t)).
$$
\end{lem}

\begin{proof}
Consider the set
$$
C_{p, \al_n} := \{ (t_0, \ldots, t_{p-1}) \in [0,1]^p : \sum_{i=0}^{p-1} t_i=1 \text{ and } t_i \le 1-\al_n \}.
$$
The map $C_{p, \al_n} \to \mr{Ran}(\fp, \al_n)$ sending $(t_0, \ldots, t_{p-1}) \in C_{p, \al_n}$ to the random element $y$ in $\fp$ satisfying $\bP(y=u)=t_u$ for each $u \in \fp$ is a bijection. Using this identification, one can regard $f_{p,r}$ as a function on $C_{p, \al_n}^{[n] \times [n]}$. 
As a function on $C_{p, \al_n}^{[n] \times [n]}$, $f_{p,r}$ is affine in each component. After all other components are fixed, the absolute value of this complex-valued affine function is convex in the remaining component. We may therefore maximize successively in each component and obtain a maximizer $(y_{kl})$ for which every $y_{kl}$ is an extreme point of $C_{p, \al_n}$.

For all sufficiently large $n$, we have $\al_n<\frac{1}{2}$. Under this assumption, an extreme point $y = (y_0, \ldots, y_{p-1}) \in C_{p, \al_n}$ satisfies $y_u=1-\al_n$ and $y_v = \al_n$ for two distinct elements $u,v \in \fp$, and all other coordinates vanish.
The random element $x'$ corresponding to $y$ is given by 
$$
\bP(x' = w) = \begin{cases}
1 - \al_n & \text{if } w=u, \\
\al_n & \text{if } w=v, \\
0 & \text{otherwise}.
\end{cases}
$$
Thus, there are distinct elements $u_{kl}, v_{kl} \in \fp$ for all $k,l \in [n]$ such that
$\lt f_{p,r}((x_{kl})) \rt \le \lt f_{p,r}((x'_{kl})) \rt$, where each random element $x'_{kl}$ in $\mr{Ran}(\fp, \al_n)$ is defined by
$$
\bP(x'_{kl} = w) = \begin{cases}
1 - \al_n & \text{if } w=u_{kl}, \\
\al_n & \text{if } w=v_{kl}, \\
0 & \text{otherwise}.
\end{cases}
$$

Since the absolute value
$$
\left | \prod_{a=1}^{i} \prod_{b=1}^{j} \bE(\ze^{x'_{k_a l_b} (g_{k_a} \cdot h_{l_b})}) \right |
$$
remains unchanged when $(u_{kl}, v_{kl})$ is replaced by $(0, v_{kl}-u_{kl})$, it follows that
$$
\lt f_{p,r}((x'_{kl})) \rt \le F_{p,r}((x'_{kl})) = F_{p,r}(X(t)),
$$
where $t = (t_{kl}) \in (\fp^*)^{[n] \times [n]}$ is defined by $t_{kl}=v_{kl}-u_{kl}$ for every $k,l \in [n]$.
\end{proof}

Write $F(t) := F_{p,r}(X(t))$ for simplicity. By Lemma \ref{lem2a}, it suffices to prove
$$
\lim_{n \to \infty}\sup_{t \in (\fp^*)^{[n] \times [n]}}F(t)=0
$$
to establish Theorem \ref{thm1_main2}. Write
$$
c(m) := \lt 1-\al_n+\al_n \ze^m \rt
$$ 
for each $m \in \fp$. For a tuple $(g_{k_a}) = (g_{k_1}, \ldots, g_{k_i}) \in (G^*)^i$, $h \in G$ and $l \in [n]$, define
$$
E_l^{t}((g_{k_a}), h) := \prod_{a=1}^{i} c(t_{k_a l}(g_{k_a} \cdot h)).
$$
For tuples $(g_{k_a}) = (g_{k_1}, \ldots, g_{k_i}) \in (G^*)^i$ and $(h_{l_b}) = (h_{l_1}, \ldots, h_{l_j}) \in (G^*)^j$, define
$$
E^{t}((g_{k_a}), (h_{l_b})) := \prod_{b=1}^{j} E_{l_b}^{t}((g_{k_a}), h_{l_b}) = \prod_{a=1}^{i} \prod_{b=1}^{j} c(t_{k_a l_b}(g_{k_a} \cdot h_{l_b})).
$$
Then
\begin{align*}
F(t) & = \frac{1}{p^{rn}}\sum_{i=r}^{n} \sum_{j=1}^{n} \sum_{\substack{1 \le k_1 < \cdots < k_i \le n \\ 1 \le l_1 < \cdots < l_j \le n}} \sum_{\substack{g_{k_1}, \ldots, g_{k_i} \in G^* \\ \left< g_{k_1}, \ldots, g_{k_i} \right> = G}}
\sum_{h_{l_1}, \ldots, h_{l_j} \in G^*} \left | \prod_{a=1}^{i} \prod_{b=1}^{j} (1-\al_n + \al_n \ze^{t_{k_a l_b}(g_{k_a} \cdot h_{l_b})}) \right | \\
& = \frac{1}{p^{rn}}\sum_{i=r}^{n} \sum_{j=1}^{n} \sum_{\substack{1 \le k_1 < \cdots < k_i \le n \\ 1 \le l_1 < \cdots < l_j \le n}} \sum_{\substack{g_{k_1}, \ldots, g_{k_i} \in G^* \\ \left< g_{k_1}, \ldots, g_{k_i} \right> = G}}
\sum_{h_{l_1}, \ldots, h_{l_j} \in G^*} E^{t}((g_{k_a}), (h_{l_b})) \\
& = \frac{1}{p^{rn}}\sum_{i=r}^{n} \sum_{1 \le k_1 < \cdots < k_i \le n} \sum_{\substack{g_{k_1}, \ldots, g_{k_i} \in G^* \\ \left< g_{k_1}, \ldots, g_{k_i} \right> = G}} \left ( \prod_{l=1}^{n} \left ( 1 + \sum_{h_l \in G^*} E_l^{t}((g_{k_a}), h_{l}) \right ) - 1 \right ).
\end{align*}

Since $F(t)$ is a complicated sum involving the products $E^{t}((g_{k_a}), (h_{l_b}))$, its behavior depends heavily on the support sizes $i$ and $j$ and the elements $g_{k_a}, h_{l_b} \in G^*$. In particular, the product $E^{t}((g_{k_a}), (h_{l_b}))$ becomes smaller as the number
$$
N := \left | \{ (a,b) \in [i] \times [j] : g_{k_a} \cdot h_{l_b} \ne 0 \} \right |
$$
increases. This motivates the decomposition of $F(t)$ into partial sums, each corresponding to a distinct range of $(i,j)$ which can then be estimated separately.

Write $[n]_k := \{ k, k+1, \ldots, n \}$ for $0 \le k \le n$. For $I,J \subseteq [n]_0$, let
$$
F^{I \times J}(t) := \frac{1}{p^{rn}}\sum_{i \in I} \sum_{j \in J} \sum_{\substack{1 \le k_1 < \cdots < k_i \le n \\ 1 \le l_1 < \cdots < l_j \le n}} \sum_{\substack{g_{k_1}, \ldots, g_{k_i} \in G^* \\ \left< g_{k_1}, \ldots, g_{k_i} \right> = G}}
\sum_{h_{l_1}, \ldots, h_{l_j} \in G^*} E^{t}((g_{k_a}), (h_{l_b}))
$$
(so $F(t) = F^{[n]_r \times [n]}(t)$) and
$$
\ol{F}^{I \times J}(t) := \frac{1}{p^{rn}}\sum_{i \in I} \sum_{j \in J} \sum_{\substack{1 \le k_1 < \cdots < k_i \le n \\ 1 \le l_1 < \cdots < l_j \le n}} \sum_{g_{k_1}, \ldots, g_{k_i} \in G^*}
\sum_{h_{l_1}, \ldots, h_{l_j} \in G^*} E^{t}((g_{k_a}), (h_{l_b})).
$$
Every summand is nonnegative because absolute values have already been taken. Hence $F^{I \times J}(t) \le \ol{F}^{I \times J}(t)$ after dropping the condition $\left< g_{k_1}, \ldots, g_{k_i} \right> = G$. The sum $\ol{F}$ is easier to use because, after dropping the generating condition, its sums over the remaining coordinates can be factorized.

Let $c_1 := \frac{1+c}{2}$ and $c_2 := \frac{1+c_1}{2}$, so that $c>c_1>c_2>1$. Let $\ga_1 := \frac{c_1-1}{2c_1^2}>0$ and $\ga_2 \in (0, \frac{1}{2})$ be a sufficiently small constant that will be specified later (in the proof of Theorem \ref{thm21_main2b}).
Let $I_1=[1, \ga_1 \frac{n}{\log n} )$, $I_2=[\ga_1 \frac{n}{\log n}, \ga_2 n)$ and $I_3=[\ga_2 n, n]$, where each interval is understood as its intersection with $\bZ$. Then
\begin{equation*}
\begin{split}
F(t) & \le F^{I_1 \times [n]}(t) + F^{[n]_r \times I_1}(t) + F^{I_2 \times [n]}(t) + F^{[n]_r \times I_2}(t) + F^{I_3 \times I_3}(t) \\
& \le \ol{F}^{I_1 \times [n]_0}(t) + \ol{F}^{[n]_0 \times I_1}(t) + \ol{F}^{I_2 \times [n]_0}(t) + \ol{F}^{[n]_0 \times I_2}(t) + F^{I_3 \times I_3}(t) \\
& = \ol{F}^{I_1 \times [n]_0}(t)+\ol{F}^{I_1 \times [n]_0}(t^{\vee})+\ol{F}^{I_2 \times [n]_0}(t)+\ol{F}^{I_2 \times [n]_0}(t^{\vee})+ F^{I_3 \times I_3}(t),
\end{split}
\end{equation*}
where $t^{\vee} = (t^{\vee}_{kl}) \in (\fp^*)^{[n] \times [n]}$ is defined by $t^{\vee}_{kl} := t_{lk}$. 
The first inequality follows by separating the cases according to the intervals containing $i$ and $j$; since all summands are nonnegative, overlaps cause no problem. The second inequality follows by dropping the generating condition in the first four terms, and the last equality follows by transposing $t$.
We obtain Theorem \ref{thm1_main2} by combining the following three theorems.

\begin{thm} \label{thm21_main2a}
One has $\displaystyle \lim_{n \to \infty}\sup_{t \in (\fp^*)^{[n] \times [n]}} \ol{F}^{I_1 \times [n]_0}(t) = 0$.
\end{thm}

\begin{thm} \label{thm21_main2b}
One has $\displaystyle \lim_{n \to \infty}\sup_{t \in (\fp^*)^{[n] \times [n]}} \ol{F}^{I_2 \times [n]_0}(t) = 0$.
\end{thm}

\begin{thm} \label{thm21_main2c}
One has $\displaystyle \lim_{n \to \infty}\sup_{t \in (\fp^*)^{[n] \times [n]}} F^{I_3 \times I_3}(t) = 0$.
\end{thm}

We prove the first two theorems in Section \ref{Sub22}, and prove Theorem \ref{thm21_main2c} in Sections \ref{Sub23} and \ref{Sub24}. When $i \in I_1$ or $j \in I_1$, factorization gives a power saving strong enough that the sum over all possible supports and nonzero values tends to zero. When $i \in I_2$ or $j \in I_2$, sufficiently many nonzero pairings give a uniform bound for each summand, while an entropy estimate controls the number of summands. When $i,j \in I_3$, these estimates are no longer sufficient because many pairings may vanish, so we use the sets $A$ and $B$ of popular values and their orthogonality structure.

\subsection{Proof of Theorems \ref{thm21_main2a} and \ref{thm21_main2b}} \label{Sub22}

We begin with a simple lemma providing an upper bound for 
$$
c(m) := \lt 1-\al_n+\al_n \ze^m \rt.
$$

\begin{lem} \label{lem2b}
For every $m \in \fp$, 
$$
c(m) \le 1 - \frac{2c_1 \log n}{n} \sin^2 \frac{\pi m}{p}
$$ for all sufficiently large $n$. 
Moreover, if $m \ne 0$, then 
$$
c(m) \le 1 - \frac{\log n}{n} \frac{8c_1}{p^2}.
$$
\end{lem}

\begin{proof}
We represent each $m\in\fp$ by the unique integer in $\{0,1,\ldots,p-1\}$. Let $\theta_m = \frac{\pi m}{p}$. A direct computation implies that
\begin{align*}
c(m) & = \left | 1 - \al_n + \al_n(\cos 2 \theta_m + \iu \sin 2 \theta_m) \right | \\
& = \sqrt{\left( 1 - \al_n + \al_n \cos 2 \theta_m \right)^2 + \al_n^2 \sin^2 2 \theta_m } \\
& = \sqrt{ 1 - 4 \al_n(1-\al_n)\sin^2 \theta_m} \\
& \le 1 - 2\al_n(1-\al_n)\sin^2 \theta_m.
\end{align*}
The first three lines follow by expanding the squared modulus, and the last inequality uses $\sqrt{1-x}\le 1-\frac{x}{2}$ for $0\le x\le 1$.
If $n$ is sufficiently large, then 
$$
\al_n(1-\al_n) = \frac{c \log n}{n} (1 - \frac{c \log n}{n}) > \frac{c_1 \log n}{n}
$$
so $c(m) \le 1 - \frac{2c_1 \log n}{n} \sin^2 \theta_m$.
If $m \ne 0$, then $\sin\frac{\pi m}{p} \ge \sin\frac{\pi}{p} \ge \frac{2}{p}$ so
\begin{equation*}
c(m) \le 1 - \frac{2c_1 \log n}{n} \sin^2 \frac{\pi m}{p}
\le 1 - \frac{2c_1 \log n}{n} \left( \frac{2}{p} \right)^2 = 1 - \frac{\log n}{n} \frac{8c_1}{p^2}. \qedhere
\end{equation*}
\end{proof}

\begin{proof}[Proof of Theorem \ref{thm21_main2a}]
Write $m_{kl} := t_{kl}(g_k \cdot h_l) \in \fp$ for simplicity. 
Then
$$
\ol{F}^{I_1 \times [n]_0}(t) = \frac{1}{p^{rn}}\sum_{i \in I_1} \sum_{1 \le k_1 < \cdots < k_i \le n} \sum_{g_{k_1}, \ldots, g_{k_i} \in G^*} \prod_{l=1}^{n} \left ( 1 + \sum_{h_l \in G^*} \prod_{a=1}^{i} c(m_{k_a l}) \right ).
$$
Assume that $n$ is sufficiently large so that $c(m) \le 1 - \frac{2c_1 \log n}{n} \sin^2 \frac{\pi m}{p}$ for every $m \in \fp$. If $i \in I_1$, then $i < \ga_1 \frac{n}{\log n}$ so
\begin{align*}
\prod_{a=1}^{i} c(m_{k_a l})
& \le \prod_{a=1}^{i} \left (1-\frac{2c_1 \log n}{n} \sin^2 \frac{\pi m_{k_a l}}{p} \right ) \\
& \le \exp \left ( - \frac{2c_1 \log n}{n} \sum_{a=1}^{i}\sin^2 \frac{\pi m_{k_a l}}{p} \right ) \\
& \le 1 - \frac{2c_1 \log n}{n} \sum_{a=1}^{i}\sin^2 \frac{\pi m_{k_a l}}{p} + \frac{1}{2} \left( \frac{2c_1 \log n}{n} \sum_{a=1}^{i}\sin^2 \frac{\pi m_{k_a l}}{p} \right)^2 \\
& \le 1 - \frac{2c_1 \log n}{n} \left( 1 - \frac{c_1 \log n}{n} i \right) \sum_{a=1}^{i}\sin^2 \frac{\pi m_{k_a l}}{p} \\
& \le 1 - \frac{2c_1 \log n}{n} (1 - c_1 \ga_1) \sum_{a=1}^{i}\sin^2 \frac{\pi m_{k_a l}}{p} \\
& = 1 - \frac{2c_2 \log n}{n} \sum_{a=1}^{i}\sin^2 \frac{\pi m_{k_a l}}{p}.
\end{align*}
(The last equality holds since $c_1(1-c_1 \ga_1)=c_2$.) The formula
$$
\sum_{h_l \in G^*}\sin^2 \frac{\pi m_{k_a l}}{p} = p^{r-1} \sum_{u=1}^{p-1} \sin^2 \frac{\pi u}{p} = \frac{p^r}{2}
$$
implies that
\begin{align*}
1 + \sum_{h_l \in G^*} \prod_{a=1}^{i} c(m_{k_a l})
\le \, & 1 + \sum_{h_l \in G^*} \left ( 1 - \frac{2c_2 \log n}{n} \sum_{a=1}^{i}\sin^2 \frac{\pi m_{k_a l}}{p} \right ) \\
= \, & p^r - \frac{2c_2 \log n}{n} \sum_{a=1}^{i}\sum_{h_l \in G^*}\sin^2 \frac{\pi m_{k_a l}}{p} \\
= \, & p^r - \frac{2c_2 \log n}{n} \sum_{a=1}^{i} \frac{p^r}{2} \\
= \, & p^r \left ( 1 - \frac{c_2 \log n}{n} i \right ).
\end{align*}
Now we have
\begin{align*}
\ol{F}^{I_1 \times [n]_0}(t)
& \le \frac{1}{p^{rn}}\sum_{i \in I_1} \sum_{1 \le k_1 < \cdots < k_i \le n} \sum_{g_{k_1}, \ldots, g_{k_i} \in G^*} (p^r)^n \left ( 1 - \frac{c_2 \log n}{n} i \right )^n \\
& \le \sum_{i \in I_1} \binom{n}{i} (p^r-1)^i \frac{1}{n^{c_2 i}} \\
& < \left ( 1 + \frac{p^r-1}{n^{c_2}} \right )^n - 1,
\end{align*}
and the last term converges to $0$ as $n \to \infty$ since $c_2>1$.
\end{proof}

\begin{proof}[Proof of Theorem \ref{thm21_main2b}]
Write $m_{kl} := t_{kl}(g_k \cdot h_l) \in \fp$, and assume that $i \in I_2$. For $g_{k_1}, \ldots, g_{k_i} \in G^*$ and $h \in G^*$, define
$$
N_h = N_h(g_{k_1}, \ldots, g_{k_i}) := \lt \{ a \in [i] : g_{k_a} \cdot h \ne 0 \} \rt.
$$
Then 
$$
\sum_{h \in G^*} N_h = \sum_{a=1}^{i} \lt \{ h \in G^* : g_{k_a} \cdot h \ne 0 \} \rt = i(p^r - p^{r-1})
$$ 
so
\begin{equation*}
\max_{h \in G^*} N_h \ge \frac{i(p^r - p^{r-1})}{p^r-1}
> \frac{\ga_1 (p-1)}{p} \frac{n}{\log n}. 
\end{equation*}
Choose $h_0 \in G^*$ for which $N_{h_0}$ is maximal. By Lemma \ref{lem2b}, we apply the nontrivial bound to the term indexed by $h_0$ and the trivial bound $1$ to all other terms. Thus
\begin{align*}
1 + \sum_{h_l \in G^*} \prod_{a=1}^{i} c(m_{k_a l}) 
& \le 1 + (p^r-2) + \left ( 1 - \frac{\log n}{n} \frac{8c_1}{p^2} \right)^{N_{h_0}} \\
& < (p^r-1)+ \left( 1 - \frac{\log n}{n} \frac{8c_1}{p^2} \right)^{\frac{\ga_1 (p-1)}{p} \frac{n}{\log n}} \\
& \le (p^r-1) + \exp \left( -\frac{8(p-1)c_1 \ga_1}{p^3} \right) \\
& =: K.
\end{align*}
The constant $K$ depends on $p$, $r$, $c_1$ and $\ga_1$, but it is independent of $n$. Moreover, the inequality $K<p^r$ provides the essential saving, while the lower bound $i\ge \ga_1 n/\log n$ ensures that this saving is strong enough when the estimate is repeated over the $n$ columns below. The above inequality implies that
\begin{align*}
\ol{F}^{I_2 \times [n]_0}(t) 
& = \frac{1}{p^{rn}}\sum_{i \in I_2} \sum_{1 \le k_1 < \cdots < k_i \le n} \sum_{g_{k_1}, \ldots, g_{k_i} \in G^*} \prod_{l=1}^{n} \left ( 1 + \sum_{h_l \in G^*} \prod_{a=1}^{i} c(m_{k_a l}) \right ) \\
& \le \frac{1}{p^{rn}}\sum_{1 \le i < \ga_2 n} \sum_{1 \le k_1 < \cdots < k_i \le n} \sum_{g_{k_1}, \ldots, g_{k_i} \in G^*} K^n \\
& \le \frac{K^n}{p^{rn}} \sum_{1 \le i < \ga_2 n} \binom{n}{i} (p^r-1)^i \\
& \le \frac{K^n}{p^{rn}} \cdot \ga_2 n \binom{n}{\lf \ga_2 n \rf} (p^r-1)^{\ga_2 n}.
\end{align*}
(The last inequality uses the assumption $\ga_2 < \frac{1}{2}$.) By \cite[Example 11.1.3]{CT06}, 
$$
\binom{n}{\lf \ga_2 n \rf} \le 2^{nH(\ga_2)} = \frac{1}{(\ga_2^{\ga_2}(1-\ga_2)^{1-\ga_2})^n},
$$
where $H(q) := - q \log_2 q - (1-q) \log_2 (1-q)$ denotes the binary entropy function. Since $\frac{K}{p^r}<1$ and
$$
\lim_{\ga_2 \to 0^+} \frac{(p^r-1)^{\ga_2}}{\ga_2^{\ga_2}(1-\ga_2)^{1-\ga_2}} = 1,
$$
we may choose $\ga_2$ sufficiently small so that
\begin{equation*}
\frac{K}{p^r} \cdot \frac{(p^r-1)^{\ga_2}}{\ga_2^{\ga_2}(1-\ga_2)^{1-\ga_2}} < 1.
\end{equation*}
In this case,
$$
\frac{K^n}{p^{rn}} \cdot \ga_2 n \binom{n}{\lf \ga_2 n \rf} (p^r-1)^{\ga_2 n}
\le \ga_2 n \left ( \frac{K}{p^r} \cdot \frac{(p^r-1)^{\ga_2}}{\ga_2^{\ga_2}(1-\ga_2)^{1-\ga_2}} \right )^n,
$$
and the right-hand side converges to $0$ as $n \to \infty$.
\end{proof}

\subsection{Proof of Theorem \ref{thm21_main2c}: Cases \texorpdfstring{$\mc{C}_1$}{C1} and \texorpdfstring{$\mc{C}_2$}{C2}} \label{Sub23}

In the remainder of the paper, we prove Theorem \ref{thm21_main2c}. When $r=1$, the proof is straightforward as follows. Indeed, if $r=1$, then $g_{k_a} \cdot h_{l_b} \ne 0$ for every $g_{k_a}, h_{l_b} \in \fp^*$. Hence
\begin{align*}
F^{I_3 \times I_3}(t)
& \le \ol{F}^{I_3 \times I_3}(t) \\
& = \frac{1}{p^n}\sum_{i,j \in I_3} \sum_{\substack{1 \le k_1 < \cdots < k_i \le n \\ 1 \le l_1 < \cdots < l_j \le n}} \sum_{\substack{g_{k_1}, \ldots, g_{k_i} \in \fp^* \\ h_{l_1}, \ldots, h_{l_j} \in \fp^*}} E^{t}((g_{k_a}), (h_{l_b})) \\
& \le \frac{1}{p^n} \sum_{i,j \in I_3} \sum_{\substack{1 \le k_1 < \cdots < k_i \le n \\ 1 \le l_1 < \cdots < l_j \le n}} \sum_{\substack{g_{k_1}, \ldots, g_{k_i} \in \fp^* \\ h_{l_1}, \ldots, h_{l_j} \in \fp^*}} \left( 1 - \frac{\log n}{n} \frac{8c_1}{p^2} \right)^{ij} \\
& \le p^n \left( 1 - \frac{\log n}{n} \frac{8c_1}{p^2} \right)^{\ga_2^2 n^2},
\end{align*}
and the last term converges to $0$ as $n \to \infty$. However, the proof of Theorem \ref{thm21_main2c} is significantly more involved when $r>1$.
The proof of Theorem \ref{thm21_main2b} used a large value of $N_h$, and hence many nonzero pairings, to obtain a fixed Fourier saving. The main difficulty is that, when $r>1$, the condition $i,j \in I_3$ no longer guarantees that the quantity
$$
N := \left | \{ (a,b) \in [i] \times [j] : g_{k_a} \cdot h_{l_b} \ne 0 \} \right |
$$
is large enough. To handle this issue, we further subdivide the case $i,j \in I_3$ into several subcases.

Let $\ga_3 < \min \left( \frac{\ga_2}{p^r-1}, \frac{1}{2p^r} \right)$ be a sufficiently small positive constant that will be specified later (in the proof of Proposition \ref{prop_case2}). For given $g_{k_1}, \ldots, g_{k_i} \in G^*$ and $h_{l_1}, \ldots, h_{l_j} \in G^*$, define
$$
i_g = i_g(g_{k_1}, \ldots, g_{k_i}) := | \{ a \in [i] : g_{k_a}=g \} |
$$
for each $g \in G^*$ and
$$
j_h = j_h(h_{l_1}, \ldots, h_{l_j}) := | \{ b \in [j] : h_{l_b}=h \} |
$$
for each $h \in G^*$. Then $i = \sum_{g \in G^*} i_g$ and $j = \sum_{h \in G^*} j_h$. Define the sets of popular values as
\begin{align*}
A & = A_{g_{k_1}, \ldots, g_{k_i}} := \{ g \in G^* : i_g \ge \ga_3 n \}, \\
B & = B_{h_{l_1}, \ldots, h_{l_j}} := \{ h \in G^* : j_h \ge \ga_3 n \}.
\end{align*}
Since $\ga_3 < \frac{\ga_2}{p^r-1}$, both $A$ and $B$ are nonempty whenever $i,j \in I_3$. 

Suppose that $g\cdot h=0$ for every $(g,h)\in A\times B$, and let $V=\left<A\right>$ and $W=\left<B\right>$. Then $W\subseteq V^\perp$, while $|A|+1\le |V|$ and $|B|+1\le |W|$, so
$$
(|A|+1)(|B|+1)\le |V||W|\le |V||V^\perp|=p^r.
$$
Consequently, for every $g_{k_1}, \ldots, g_{k_i} \in G^*$ and $h_{l_1}, \ldots, h_{l_j} \in G^*$, exactly one of the following holds:
\begin{enumerate}
\item $g \cdot h \ne 0$ for some $(g,h) \in A \times B$;
\item $g \cdot h = 0$ for every $(g,h) \in A \times B$, and $(\lt A \rt+1)(\lt B \rt+1) < p^r$;
\item $g \cdot h = 0$ for every $(g,h) \in A \times B$, and $(\lt A \rt+1)(\lt B \rt+1) = p^r$.
\end{enumerate}

Let $\mc{C} := \mc{C}(g_{k_1}, \ldots, g_{k_i}, h_{l_1}, \ldots, h_{l_j})$ be a condition on the elements $g_{k_1}, \ldots, g_{k_i}, h_{l_1}, \ldots, h_{l_j} \in G^*$. Denote the three conditions above by $\mc{C}_1$, $\mc{C}_2$ and $\mc{C}_3$, respectively. For a condition $\mc{C}$ and $t \in (\fp^*)^{[n] \times [n]}$, define
$$
F_{\mc{C}}^{I_3 \times I_3}(t) := \frac{1}{p^{rn}}\sum_{i,j \in I_3} \sum_{\substack{1 \le k_1 < \cdots < k_i \le n \\ 1 \le l_1 < \cdots < l_j \le n}} \sum_{\substack{g_{k_1}, \ldots, g_{k_i} \in G^* \\ \left< g_{k_1}, \ldots, g_{k_i} \right> = G}}
\sum_{\substack{h_{l_1}, \ldots, h_{l_j} \in G^* \\ \mc{C} \text{ is satisfied}}} E^{t}((g_{k_a}), (h_{l_b})).
$$
Then
$$
F^{I_3 \times I_3}(t) = \sum_{k=1}^{3} F_{\mc{C}_k}^{I_3 \times I_3}(t).
$$
As explained in the last paragraph of the introduction, the three cases are treated by different estimates, and the case $\mc{C}_3$ requires a further subdivision. In this section, we prove
$$
\lim_{n \to \infty}\sup_{t \in (\fp^*)^{[n] \times [n]}} F_{\mc{C}_1}^{I_3 \times I_3}(t)
=
\lim_{n \to \infty}\sup_{t \in (\fp^*)^{[n] \times [n]}} F_{\mc{C}_2}^{I_3 \times I_3}(t)=0.
$$
In Section \ref{Sub24}, we prove
$$
\lim_{n \to \infty}\sup_{t \in (\fp^*)^{[n] \times [n]}} F_{\mc{C}_3}^{I_3 \times I_3}(t)=0
$$
by subdividing $\mc{C}_3$ into three subcases $\mc{C}_{3,a}$, $\mc{C}_{3,b}$ and $\mc{C}_{3,c}$.

\begin{prop} \label{prop_case1}
One has $\displaystyle \lim_{n \to \infty}\sup_{t \in (\fp^*)^{[n] \times [n]}} F_{\mc{C}_1}^{I_3 \times I_3}(t) = 0$.
\end{prop}

\begin{proof}
If the condition $\mc{C}_1$ holds, then $g \cdot h \ne 0$ for some $(g,h) \in A \times B$ so
$$
E^{t}((g_{k_a}), (h_{l_b}))
\le \prod_{g_{k_a}=g} \prod_{h_{l_b}=h} c(t_{k_a l_b}(g \cdot h))
\le \left( 1 - \frac{\log n}{n} \frac{8c_1}{p^2} \right)^{i_g j_h} 
\le \left( 1 - \frac{\log n}{n} \frac{8c_1}{p^2} \right)^{\ga_3^2 n^2}
$$
by Lemma \ref{lem2b}. Therefore
\begin{align*}
F_{\mc{C}_1}^{I_3 \times I_3}(t)
& \le \frac{1}{p^{rn}}\sum_{i,j \in I_3} \sum_{\substack{1 \le k_1 < \cdots < k_i \le n \\ 1 \le l_1 < \cdots < l_j \le n}} \sum_{\substack{g_{k_1}, \ldots, g_{k_i} \in G^* \\ \left< g_{k_1}, \ldots, g_{k_i} \right> = G}}
\sum_{h_{l_1}, \ldots, h_{l_j} \in G^*}
\left( 1 - \frac{\log n}{n} \frac{8c_1}{p^2} \right)^{\ga_3^2 n^2} \\
& \le \frac{1}{p^{rn}}\sum_{i,j \in [n]_0} \binom{n}{i} \binom{n}{j} (p^r-1)^i (p^r-1)^j
\exp \left( - \frac{\log n}{n} \frac{8c_1}{p^2} \ga_3^2 n^2 \right) \\
& = p^{rn} n^{- \frac{8c_1 \ga_3^2}{p^2} n},
\end{align*}
and the last term converges to $0$ as $n \to \infty$.
\end{proof}

\begin{prop} \label{prop_case2}
One has $\displaystyle \lim_{n \to \infty}\sup_{t \in (\fp^*)^{[n] \times [n]}} F_{\mc{C}_2}^{I_3 \times I_3}(t) = 0$.
\end{prop}

\begin{proof}
Let $S$ be the set of tuples $(A_0, B_0)$ such that $A_0$ and $B_0$ are nonempty subsets of $G^*$, $g \cdot h = 0$ for every $(g,h) \in A_0 \times B_0$ and $(\lt A_0 \rt+1)(\lt B_0 \rt+1) < p^r$. 
Since $0\le E^t\le 1$, we may first drop the generating condition and then replace $E^t$ by $1$. This leaves only the number of tuples with prescribed popular sets. Thus
\begin{align*}
F_{\mc{C}_2}^{I_3 \times I_3}(t) & = \sum_{(A_0,B_0) \in S} \frac{1}{p^{rn}}\sum_{i,j \in I_3} \sum_{\substack{1 \le k_1 < \cdots < k_i \le n \\ 1 \le l_1 < \cdots < l_j \le n}} \sum_{\substack{g_{k_1}, \ldots, g_{k_i} \in G^* \\ \left< g_{k_1}, \ldots, g_{k_i} \right> = G \\ A_{g_{k_1}, \ldots, g_{k_i}} = A_0}}
\sum_{\substack{h_{l_1}, \ldots, h_{l_j} \in G^* \\ B_{h_{l_1}, \ldots, h_{l_j}} = B_0}} E^{t}((g_{k_a}), (h_{l_b})) \\
& \le \sum_{(A_0,B_0) \in S} \frac{1}{p^{rn}}\sum_{i,j \in I_3} \sum_{\substack{1 \le k_1 < \cdots < k_i \le n \\ 1 \le l_1 < \cdots < l_j \le n}} \sum_{\substack{g_{k_1}, \ldots, g_{k_i} \in G^* \\ A_{g_{k_1}, \ldots, g_{k_i}} = A_0}}
\sum_{\substack{h_{l_1}, \ldots, h_{l_j} \in G^* \\ B_{h_{l_1}, \ldots, h_{l_j}} = B_0}} 1 \\
& = \frac{1}{p^{rn}} \sum_{(A_0,B_0) \in S}
\left( \sum_{i \in I_3} \sum_{1 \le k_1 < \cdots < k_i \le n} 
\lt \left\{ (g_{k_1}, \ldots, g_{k_i}) \in (G^*)^i : A_{g_{k_1}, \ldots, g_{k_i}} = A_0 \right\} \rt \right) \\
& \;\;\; \left( \sum_{j \in I_3} \sum_{1 \le l_1 < \cdots < l_j \le n} 
\lt \left\{ (h_{l_1}, \ldots, h_{l_j}) \in (G^*)^j : B_{h_{l_1}, \ldots, h_{l_j}} = B_0 \right\} \rt \right) \\
& = \frac{1}{p^{rn}} \sum_{(A_0,B_0) \in S}
\left( \sum_{i \in I_3} \binom{n}{i} 
\lt \left\{ (g_{1}, \ldots, g_{i}) \in (G^*)^i : A_{g_{1}, \ldots, g_{i}} = A_0 \right\} \rt \right) \\
& \;\;\; \left( \sum_{j \in I_3} \binom{n}{j} 
\lt \left\{ (h_{1}, \ldots, h_{j}) \in (G^*)^j : B_{h_{1}, \ldots, h_{j}} = B_0 \right\} \rt \right).
\end{align*}
The last equality holds since the sizes of the sets
$$
\left\{ (g_{k_1}, \ldots, g_{k_i}) \in (G^*)^i : A_{g_{k_1}, \ldots, g_{k_i}} = A_0 \right\} \text{ and } \left\{ (h_{l_1}, \ldots, h_{l_j}) \in (G^*)^j : B_{h_{l_1}, \ldots, h_{l_j}} = B_0 \right\}
$$
do not depend on the indices $k_1, \ldots, k_i$ and $l_1, \ldots, l_j$.

Write $\ga_4 := p^r \ga_3$. Then $\ga_4 < \frac{1}{2}$ by the assumption $\ga_3 < \frac{1}{2p^r}$. 
Fix an element $(A_0,B_0) \in S$ and write $A_0^c := G^* \setminus A_0$ and $B_0^c := G^* \setminus B_0$.
If $(g_{1}, \ldots, g_{i}) \in (G^*)^i$ satisfies $A_{g_{1}, \ldots, g_{i}} = A_0$, then $i_g \ge \ga_3 n$ for every $g \in A_0$ and $i_g < \ga_3 n$ for every $g \in A_0^c$. To count such tuples, set $i_{\mr{sm}}=\sum_{g\in A_0^c}i_g$. The first two inequalities below separate the occurrences of the exceptional values in $A_0^c$ from those of the popular values in $A_0$; the corresponding multinomial sums are bounded by $|A_0^c|^{i_{\mr{sm}}}$ and $|A_0|^{i-i_{\mr{sm}}}$, respectively. Therefore
\begin{align*}
& \lt \left\{ (g_{1}, \ldots, g_{i}) \in (G^*)^i : A_{g_{1}, \ldots, g_{i}} = A_0 \right\} \rt \\
= \, & \sum_{\substack{\sum_{g \in G^*} i_g = i \\ i_g \ge \ga_3 n \text{ for all } g \in A_0 \\ i_g < \ga_3 n \text{ for all } g \in A_0^c}} \frac{i!}{\prod_{g \in G^*} i_g!} \\
\le \, & \sum_{\substack{(i_g)_{g \in A_0^c} \\ i_g < \ga_3 n \text{ for all } g \in A_0^c}} \frac{\left( \sum_{g \in A_0^c} i_g \right)!}{\prod_{g \in A_0^c} i_g!} 
\sum_{\substack{(i_g)_{g \in A_0} \\ \sum_{g \in A_0} i_g = i - \sum_{g \in A_0^c} i_g}} \frac{i!}{\left( \sum_{g \in A_0^c} i_g \right)! \prod_{g \in A_0} i_g!} \\
\le \, & \sum_{i_{\mr{sm}} < \lt A_0^c \rt \ga_3 n} \lt A_0^c \rt^{i_{\mr{sm}}} 
\sum_{\substack{(i_g)_{g \in A_0} \\ \sum_{g \in A_0} i_g = i - i_{\mr{sm}}}} \frac{i!}{i_{\mr{sm}}! \prod_{g \in A_0} i_g!} \\
\le \, & \sum_{i_{\mr{sm}} < p^r \ga_3 n} (p^r)^{i_{\mr{sm}}} \binom{i}{i_{\mr{sm}}}
\sum_{\substack{(i_g)_{g \in A_0} \\ \sum_{g \in A_0} i_g = i - i_{\mr{sm}}}} \frac{(i-i_{\mr{sm}})!}{\prod_{g \in A_0} i_g!} \\
\le \, & \sum_{i_{\mr{sm}} < \ga_4 n} (p^r)^{i_{\mr{sm}}} \binom{n}{i_{\mr{sm}}}
\lt A_0 \rt^{i-i_{\mr{sm}}} \\
\le \, & p^{r \ga_4 n} \cdot (\ga_4 n+1) \binom{n}{\lf \ga_4 n \rf} \lt A_0 \rt^{i}.
\end{align*}
In the last step, since $\ga_4<\frac12$, we bound the sum by its largest binomial coefficient multiplied by the number of terms, which is at most $\ga_4 n+1$.

By \cite[Example 11.1.3]{CT06}, 
$$
\binom{n}{\lf \ga_4 n \rf} \le \frac{1}{(\ga_4^{\ga_4}(1-\ga_4)^{1-\ga_4})^n}
$$
so we have
\begin{align*}
& \sum_{i \in I_3} \binom{n}{i} 
\lt \left\{ (g_{1}, \ldots, g_{i}) \in (G^*)^i : A_{g_{1}, \ldots, g_{i}} = A_0 \right\} \rt \\
\le \, & \sum_{i \in I_3} \binom{n}{i} p^{r \ga_4 n} \cdot (\ga_4 n+1) \frac{1}{(\ga_4^{\ga_4}(1-\ga_4)^{1-\ga_4})^n} \lt A_0 \rt^{i} \\
\le \, & (\lt A_0 \rt+1)^n \left ( \frac{p^{r \ga_4}}{\ga_4^{\ga_4}(1-\ga_4)^{1-\ga_4}} \right)^n (\ga_4 n+1).
\end{align*}
For the same reason, 
\begin{align*}
\sum_{j \in I_3} \binom{n}{j} 
\lt \left\{ (h_{1}, \ldots, h_{j}) \in (G^*)^j : B_{h_{1}, \ldots, h_{j}} = B_0 \right\} \rt 
\le (\lt B_0 \rt+1)^n \left ( \frac{p^{r \ga_4}}{\ga_4^{\ga_4}(1-\ga_4)^{1-\ga_4}} \right)^n (\ga_4 n+1)
\end{align*}
so the condition $(\lt A_0 \rt+1)(\lt B_0 \rt+1) < p^r$ implies that
\begin{equation} \label{eq23a}
\begin{split}
F_{\mc{C}_2}^{I_3 \times I_3}(t)
& \le \frac{1}{p^{rn}} \sum_{(A_0, B_0) \in S} (\lt A_0 \rt+1)^n (\lt B_0 \rt+1)^n \left ( \frac{p^{r \ga_4}}{\ga_4^{\ga_4}(1-\ga_4)^{1-\ga_4}} \right)^{2n} (\ga_4 n+1)^2 \\ 
& \le \lt S \rt \frac{(p^r-1)^n}{p^{rn}} \left ( \frac{p^{r \ga_4}}{\ga_4^{\ga_4}(1-\ga_4)^{1-\ga_4}} \right)^{2n} (\ga_4 n+1)^2.
\end{split}
\end{equation}
Now we may choose $\ga_3$ sufficiently small so that $\ga_4 = p^r \ga_3$ satisfies 
$$
\frac{p^r-1}{p^r} \cdot \left ( \frac{p^{r \ga_4}}{\ga_4^{\ga_4}(1-\ga_4)^{1-\ga_4}} \right )^2 < 1.
$$
In this case, the last term of \eqref{eq23a} converges to $0$ as $n \to \infty$. 
\end{proof}

\subsection{Proof of Theorem \ref{thm21_main2c}: Case \texorpdfstring{$\mc{C}_3$}{C3}} \label{Sub24}

It remains to prove
$$
\lim_{n \to \infty}\sup_{t \in (\fp^*)^{[n] \times [n]}}F_{\mc{C}_3}^{I_3 \times I_3}(t)=0.
$$
Assume that the condition $\mc{C}_3$ is satisfied, and let $V$ (resp. $W$) be the $\fp$-subspace of $G$ generated by the set $A$ (resp. $B$). Then 
$$
\lt V \rt \lt W \rt \ge (\lt A \rt + 1)(\lt B \rt + 1) = p^r
$$
and $g \cdot h = 0$ for every $g \in V$ and $h \in W$. This is possible only when $A = V \setminus \{ 0 \}$, $B = W \setminus \{ 0 \}$ and $W=V^{\perp}$, where $V^{\perp}$ denotes the orthogonal complement of $V$ in $G$. Since $A$ and $B$ are nonempty, the subspace $V$ is neither the zero subspace nor the whole space $G$.

Choose a sufficiently large constant $\ga_5>\max(\ga_1,1)$ such that 
\begin{equation} \label{eq24_ga5}
\exp \left(\frac{8c_1\ga_5 \ga_3}{p^2} \right) \ge 2p^r.
\end{equation}
Let $A^c = G^* \setminus A = G \setminus V$ and $B^c = G^* \setminus B = G \setminus V^{\perp}$. The remaining issue is how often values in $A^c$ and $B^c$ occur. In $\mc{C}_{3,a}$, one such value occurs more than $\ga_5 \frac{n}{\log n}$ times and gives an exponentially small bound. In $\mc{C}_{3,b}$, no single value occurs that frequently, but the total multiplicity of the values in $A^c$ or $B^c$ is large, which again yields an exponentially small bound. In $\mc{C}_{3,c}$, the total multiplicities of the values in both $A^c$ and $B^c$ are small, and we instead use the generating condition and average the Fourier factors over $V$ and $V^\perp$.

When $\mc{C}_3$ is satisfied, exactly one of the following holds:
\begin{enumerate}
\item[($\mc{C}_{3,a}$)] $i_g > \ga_5 \frac{n}{\log n}$ for some $g \in A^c$, or $j_h > \ga_5 \frac{n}{\log n}$ for some $h \in B^c$;

\item[($\mc{C}_{3,b}$)] $\mc{C}_{3,a}$ does not hold, and $\sum_{g \in A^c} i_g > \ga_1 \frac{n}{\log n}$ or $\sum_{h \in B^c} j_h > \ga_1 \frac{n}{\log n}$;

\item[($\mc{C}_{3,c}$)] $\sum_{g \in A^c} i_g \le \ga_1 \frac{n}{\log n}$ and $\sum_{h \in B^c} j_h \le \ga_1 \frac{n}{\log n}$.
\end{enumerate}
Then
$$
F_{\mc{C}_3}^{I_3 \times I_3}(t) = F_{\mc{C}_{3,a}}^{I_3 \times I_3}(t) + F_{\mc{C}_{3,b}}^{I_3 \times I_3}(t) + F_{\mc{C}_{3,c}}^{I_3 \times I_3}(t)
$$
so it is enough to prove that for each of the three contributions, the supremum over $t\in(\fp^*)^{[n]\times[n]}$ converges to $0$ as $n\to\infty$.

\begin{prop} \label{prop_case3a}
One has $\displaystyle \lim_{n \to \infty}\sup_{t \in (\fp^*)^{[n] \times [n]}} F_{\mc{C}_{3,a}}^{I_3 \times I_3}(t) = 0$.
\end{prop}

\begin{proof}
Assume that the condition $\mc{C}_{3,a}$ is satisfied. Suppose first that $i_g > \ga_5 \frac{n}{\log n}$ for some $g \in A^c=G\setminus V$. Since $g\notin V=(V^{\perp})^{\perp}$, there exists $h \in B=V^{\perp}\setminus\{0\}$ such that $g\cdot h\ne 0$, and the definition of $B$ gives $j_h\ge \ga_3n$. By Lemma \ref{lem2b} and \eqref{eq24_ga5},
$$
E^{t}((g_{k_a}), (h_{l_b}))
\le \left( 1 - \frac{\log n}{n} \frac{8c_1}{p^2} \right)^{i_gj_h}
\le \left( 1 - \frac{\log n}{n} \frac{8c_1}{p^2} \right)^{\ga_5 \frac{n}{\log n} \cdot \ga_3 n}
\le \exp \left(- \frac{8c_1\ga_5 \ga_3}{p^2} n \right)
\le \frac{1}{(2p^r)^n}.
$$
The same inequality holds if $j_h > \ga_5 \frac{n}{\log n}$ for some $h \in B^c=G\setminus V^{\perp}$. Now
\begin{align*}
F_{\mc{C}_{3,a}}^{I_3 \times I_3}(t)
\le \frac{1}{p^{rn}}\sum_{i,j \in I_3} \sum_{\substack{1 \le k_1 < \cdots < k_i \le n \\ 1 \le l_1 < \cdots < l_j \le n}} \sum_{g_{k_1}, \ldots, g_{k_i} \in G^*}
\sum_{\substack{h_{l_1}, \ldots, h_{l_j} \in G^*}} \frac{1}{(2p^r)^n}
\le \frac{1}{p^{rn}} (p^{rn})^2 \frac{1}{(2p^r)^n} = \frac{1}{2^n}
\end{align*}
and the bound is independent of $t$, proving the proposition.
\end{proof}

In the next two propositions, we use the following notation. Let $S$ be the set of all proper nontrivial $\fp$-subspaces of $G$. For an $\fp$-subspace $V$ of $G$, write $V^* := V \setminus \{ 0 \}$.

\begin{prop} \label{prop_case3b}
One has $\displaystyle \lim_{n \to \infty}\sup_{t \in (\fp^*)^{[n] \times [n]}} F_{\mc{C}_{3,b}}^{I_3 \times I_3}(t) = 0$.
\end{prop}

\begin{proof}
Assume that the condition $\mc{C}_{3,b}$ is satisfied. Then either $i_g > \frac{\ga_1}{p^r} \frac{n}{\log n}$ for some $g \in A^c$, or $j_h> \frac{\ga_1}{p^r} \frac{n}{\log n}$ for some $h \in B^c$. Suppose that $i_g > \frac{\ga_1}{p^r} \frac{n}{\log n}$ for some $g \in A^c$ and choose $h \in B$ such that $g \cdot h \ne 0$. By Lemma \ref{lem2b},
$$
E^{t}((g_{k_a}), (h_{l_b})) 
\le \left( 1 - \frac{\log n}{n} \frac{8c_1}{p^2} \right)^{\frac{\ga_1}{p^r} \frac{n}{\log n} \cdot \ga_3 n}
\le \exp \left( - \frac{8c_1\ga_1 \ga_3}{p^{r+2}} n \right)
= \frac{1}{K^n}
$$
for $K := \exp \left( \frac{8c_1\ga_1 \ga_3}{p^{r+2}} \right) > 1$. The same bound applies when $j_h> \frac{\ga_1}{p^r} \frac{n}{\log n}$ for some $h \in B^c$.
Thus
\begin{equation} \label{eq24_3b1}
\begin{split}
F_{\mc{C}_{3,b}}^{I_3 \times I_3}(t)
& \le \sum_{V \in S} \frac{1}{p^{rn}}\sum_{i,j \in I_3} \sum_{\substack{1 \le k_1 < \cdots < k_i \le n \\ 1 \le l_1 < \cdots < l_j \le n}} \sum_{\substack{g_{k_1}, \ldots, g_{k_i} \in G^* \\ A_{g_{k_1}, \ldots, g_{k_i}} = V^* \\ i_g \le \ga_5 \frac{n}{\log n} \text{ for all } g \in G \setminus V}}
\sum_{\substack{h_{l_1}, \ldots, h_{l_j} \in G^* \\ B_{h_{l_1}, \ldots, h_{l_j}} = (V^{\perp})^* \\ j_h \le \ga_5 \frac{n}{\log n} \text{ for all } h \in G \setminus V^{\perp}}} \frac{1}{K^n} \\
& = \frac{1}{K^n p^{rn}} \sum_{V \in S} 
\left( \sum_{i \in I_3} \sum_{1 \le k_1 < \cdots < k_i \le n} \# \left\{ \begin{matrix}
(g_{k_1}, \ldots, g_{k_i}) \in (G^*)^i : A_{g_{k_1}, \ldots, g_{k_i}} = V^* \\
\text{ and } i_g \le \ga_5 \frac{n}{\log n} \text{ for all } g \in G \setminus V \end{matrix} \right\} \right) \\
& \;\;\; 
\left( \sum_{j \in I_3} \sum_{1 \le l_1 < \cdots < l_j \le n} \# \left\{ \begin{matrix}
(h_{l_1}, \ldots, h_{l_j}) \in (G^*)^j : B_{h_{l_1}, \ldots, h_{l_j}} = (V^{\perp})^* \\
\text{ and } j_h \le \ga_5 \frac{n}{\log n} \text{ for all } h \in G \setminus V^{\perp} \end{matrix} \right\} \right).
\end{split}
\end{equation}
For all sufficiently large $n$, we have $p^r\ga_5 \frac{n}{\log n}<\frac{n}{2}$. By applying the same argument used in the proof of Proposition \ref{prop_case2} to bound the size of the set
$$
\left\{ (g_{1}, \ldots, g_{i}) \in (G^*)^i : A_{g_{1}, \ldots, g_{i}} = A_0 \right\},
$$
we deduce that
\begin{align*}
\# \left\{
\begin{matrix}
(g_{k_1}, \ldots, g_{k_i}) \in (G^*)^i : A_{g_{k_1}, \ldots, g_{k_i}} = V^* \\
\text{ and } i_g \le \ga_5 \frac{n}{\log n} \text{ for all } g \in G \setminus V
\end{matrix} \right\} 
& \le \sum_{\substack{\sum_{g \in G^*} i_g = i \\ i_g \le \ga_5 \frac{n}{\log n} \text{ for all } g \in G \setminus V}} \frac{i!}{\prod_{g \in G^*} i_g!} \\
& \le p^{r p^r \ga_5 \frac{n}{\log n}} \cdot \left( p^r \ga_5 \frac{n}{\log n}+1 \right) \binom{n}{\left \lfloor p^r \ga_5 \frac{n}{\log n} \right \rfloor} \left | V^* \right|^i
\end{align*}
for each $i \in I_3$ and $1 \le k_1 < \cdots < k_i \le n$. Hence
\begin{equation} \label{eq24_3b2}
\begin{split}
& \sum_{i \in I_3} \sum_{1 \le k_1 < \cdots < k_i \le n} \# \left\{ \begin{matrix}
(g_{k_1}, \ldots, g_{k_i}) \in (G^*)^i : A_{g_{k_1}, \ldots, g_{k_i}} = V^* \\
\text{ and } i_g \le \ga_5 \frac{n}{\log n} \text{ for all } g \in G \setminus V \end{matrix} \right\} \\
\le \, & p^{r p^r \ga_5 \frac{n}{\log n}} \cdot \left( p^r \ga_5 \frac{n}{\log n}+1 \right) \binom{n}{\left \lfloor p^r \ga_5 \frac{n}{\log n} \right \rfloor} \sum_{i=0}^{n} \binom{n}{i}\left | V^* \right|^i \\
= \, & p^{r p^r \ga_5 \frac{n}{\log n}} \cdot \left( p^r \ga_5 \frac{n}{\log n}+1 \right) \binom{n}{\left \lfloor p^r \ga_5 \frac{n}{\log n} \right \rfloor} \left | V \right|^n.
\end{split}
\end{equation}
For the same reason, we have
\begin{equation} \label{eq24_3b3}
\begin{split}
& \sum_{j \in I_3} \sum_{1 \le l_1 < \cdots < l_j \le n} \# \left\{ \begin{matrix}
(h_{l_1}, \ldots, h_{l_j}) \in (G^*)^j : B_{h_{l_1}, \ldots, h_{l_j}} = (V^{\perp})^* \\
\text{ and } j_h \le \ga_5 \frac{n}{\log n} \text{ for all } h \in G \setminus V^{\perp} \end{matrix} \right\} \\
\le \, & p^{r p^r \ga_5 \frac{n}{\log n}} \cdot \left( p^r \ga_5 \frac{n}{\log n}+1 \right) \binom{n}{\left \lfloor p^r \ga_5 \frac{n}{\log n} \right \rfloor} \left | V^{\perp} \right|^n.
\end{split}
\end{equation}

Combining \eqref{eq24_3b1}, \eqref{eq24_3b2} and \eqref{eq24_3b3}, we obtain
\begin{align*}
F_{\mc{C}_{3,b}}^{I_3 \times I_3}(t) 
& \le \frac{1}{K^n p^{rn}} \sum_{V \in S} \left( p^{r p^r \ga_5 \frac{n}{\log n}} \cdot \left( p^r \ga_5 \frac{n}{\log n}+1 \right) \binom{n}{\left \lfloor p^r \ga_5 \frac{n}{\log n} \right \rfloor} \right)^2 \left | V \right|^n\left | V^{\perp} \right|^n \\
& \le \frac{\left | S \right|}{K^n} \left( p^{r p^r \ga_5 \frac{n}{\log n}} \cdot \left( p^r \ga_5 \frac{n}{\log n}+1 \right) \left( \frac{en}{p^r \ga_5 \frac{n}{\log n}}\right)^{p^r \ga_5 \frac{n}{\log n}} \right)^2 \\
& \le \frac{\left | S \right|}{K^n} \left( p^{r p^r \ga_5 \frac{n}{\log n}} \cdot n (e \log n)^{p^r \ga_5 \frac{n}{\log n}} \right)^2 \\
& = \left | S \right| \left( \frac{p^{\frac{2r p^r \ga_5}{\log n}} n^{\frac{2}{n}} (e \log n)^{\frac{2p^r \ga_5}{\log n}}}{K} \right)^n,
\end{align*}
where the second inequality uses the bound $\binom{n}{k} \le \left( \frac{en}{k} \right)^k$ for all $n \ge k \ge 1$, and the third inequality uses $p^r \ga_5>1$ and $p^r\ga_5 \frac{n}{\log n}+1\le n$ for all sufficiently large $n$.
The last term of the above inequality converges to $0$ as $n \to \infty$, since $K>1$ and
\begin{equation*}
\lim_{n \to \infty} p^{\frac{2r p^r \ga_5}{\log n}} n^{\frac{2}{n}} (e \log n)^{\frac{2p^r \ga_5}{\log n}} = 1. \qedhere
\end{equation*}
\end{proof}

\begin{prop} \label{prop_case3c}
One has $\displaystyle \lim_{n \to \infty}\sup_{t \in (\fp^*)^{[n] \times [n]}} F_{\mc{C}_{3,c}}^{I_3 \times I_3}(t) = 0$.
\end{prop}

\begin{proof}
If $A_{g_{k_1}, \ldots, g_{k_i}} = V^*$ for a proper $\fp$-subspace $V$ of $G$ and $\left< g_{k_1}, \ldots, g_{k_i} \right> = G$, then $i_g > 0$ for some $g \in G \setminus V$. Hence
\begin{equation} \label{eq_3c1}
\begin{split}
F_{\mc{C}_{3,c}}^{I_3 \times I_3}(t)
& = \sum_{V \in S} \frac{1}{p^{rn}}\sum_{i,j \in I_3} \sum_{\substack{1 \le k_1 < \cdots < k_i \le n \\ 1 \le l_1 < \cdots < l_j \le n}} \sum_{\substack{g_{k_1}, \ldots, g_{k_i} \in G^* \\ A_{g_{k_1}, \ldots, g_{k_i}} = V^* \\ \left< g_{k_1}, \ldots, g_{k_i} \right> = G \\ \sum_{g \in G \setminus V} i_g \le \ga_1 \frac{n}{\log n} }}
\sum_{\substack{h_{l_1}, \ldots, h_{l_j} \in G^* \\ B_{h_{l_1}, \ldots, h_{l_j}} = (V^{\perp})^* \\ \sum_{h \in G \setminus V^{\perp}} j_h \le \ga_1 \frac{n}{\log n}}} E^{t}((g_{k_a}), (h_{l_b})) \\
& \le \sum_{V \in S} \frac{1}{p^{rn}}\sum_{i,j \in I_3} \sum_{\substack{1 \le k_1 < \cdots < k_i \le n \\ 1 \le l_1 < \cdots < l_j \le n}} \sum_{\substack{g_{k_1}, \ldots, g_{k_i} \in G^* \\ A_{g_{k_1}, \ldots, g_{k_i}} = V^* \\ 1 \le \sum_{g \in G \setminus V} i_g \le \ga_1 \frac{n}{\log n} }}
\sum_{\substack{h_{l_1}, \ldots, h_{l_j} \in G^* \\ B_{h_{l_1}, \ldots, h_{l_j}} = (V^{\perp})^* \\ \sum_{h \in G \setminus V^{\perp}} j_h \le \ga_1 \frac{n}{\log n}}} E^{t}((g_{k_a}), (h_{l_b})).
\end{split}
\end{equation}

Given $g_{k_1}, \ldots, g_{k_i} \in G^*$, extend them to elements $g_1, \ldots, g_n \in G$ by setting $g_k=0$ for every $k \in [n] \setminus \{ k_1, \ldots, k_i \}$. Similarly, extend $h_{l_1}, \ldots, h_{l_j} \in G^*$ to elements $h_1, \ldots, h_n \in G$ by setting $h_l=0$ for every $l \in [n] \setminus \{ l_1, \ldots, l_j \}$. After making these extensions, write $m_{kl}:=t_{kl}(g_k\cdot h_l)\in\fp$ for every $k,l\in[n]$. In this notation, we define
$$
A_{g_1, \ldots, g_n} := A_{g_{k_1}, \ldots, g_{k_i}} \, \text{ and } \,
B_{h_1, \ldots, h_n} := B_{h_{l_1}, \ldots, h_{l_j}}.
$$
For each $g,h \in G^*$, let
\begin{align*}
i_g & = i_g(g_1, \ldots, g_n) := \left | \{ k \in [n] : g_k = g \} \right |, \\
j_h & = j_h(h_1, \ldots, h_n) := \left | \{ l \in [n] : h_l=h \} \right |.
\end{align*}
Then $i_g = \left | \{ a \in [i] : g_{k_a} = g \} \right |$ and $j_h = \left | \{ b \in [j] : h_{l_b} = h \} \right |$, so the definitions of $i_g$ and $j_h$ agree with their earlier definitions for $g_{k_1}, \ldots, g_{k_i} \in G^*$ and $h_{l_1}, \ldots, h_{l_j} \in G^*$.
The last term of \eqref{eq_3c1} is equal to
\begin{equation} \label{eq_3c1.5}
\frac{1}{p^{rn}} \sum_{V \in S} \sum_{i,j \in I_3} \sum_{\substack{g_1, \ldots, g_n \in G \\ A_{g_1, \ldots, g_n} = V^* \\ \sum_{g \in G^*} i_g = i \\ 1 \le \sum_{g \in G \setminus V} i_g \le \ga_1 \frac{n}{\log n} }}
\sum_{\substack{h_1, \ldots, h_n \in G \\ B_{h_1, \ldots, h_n} = (V^{\perp})^* \\ \sum_{h \in G^*} j_h = j \\ \sum_{h \in G \setminus V^{\perp}} j_h \le \ga_1 \frac{n}{\log n}}} \prod_{k=1}^{n} \prod_{l=1}^{n} c(m_{kl}).
\end{equation}
Define
\begin{align*}
X := \{ k \in [n] : g_k \in G \setminus V \}, \,\,\, 
Y := \{ l \in [n] : h_l \in G \setminus V^{\perp} \}
\end{align*}
and write $X^c := [n] \setminus X$ and $Y^c := [n] \setminus Y$.
Then 
$$
|X| = \sum_{g \in G \setminus V} i_g \, \text{ and } \,
|Y| = \sum_{h \in G \setminus V^{\perp}} j_h
$$
so \eqref{eq_3c1.5} is bounded above by
\begin{equation} \label{eq_3c2}
\frac{1}{p^{rn}} \sum_{V \in S}
\sum_{\substack{X \subseteq [n] \\ 1 \le |X| \le \ga_1 \frac{n}{\log n}}} 
\sum_{\substack{Y \subseteq [n] \\ |Y| \le \ga_1 \frac{n}{\log n}}} 
\sum_{\substack{(g_x)_{x \in X} \in (G \setminus V)^X \\ (g_{x'})_{x' \in X^c} \in V^{X^c}}}
\sum_{\substack{(h_y)_{y \in Y} \in (G \setminus V^{\perp})^Y \\ (h_{y'})_{y' \in Y^c} \in (V^{\perp})^{Y^c}}}
\prod_{k=1}^{n} \prod_{l=1}^{n} c(m_{kl}).
\end{equation}

First we provide an upper bound for $\prod_{k=1}^{n} \prod_{l=1}^{n} c(m_{kl})$. From the inequality
$$
\prod_{k=1}^{n} \prod_{l=1}^{n} c(m_{kl}) \le \prod_{x \in X} \prod_{y' \in Y^c} c(m_{xy'}) \cdot \prod_{y \in Y} \prod_{x' \in X^c} c(m_{x'y}),
$$
it suffices to bound $\prod_{y \in Y} c(m_{x'y})$ for each $x' \in X^c$ and $\prod_{x \in X} c(m_{xy'})$ for each $y' \in Y^c$.
Given $x' \in X^c$, we have $g_{x'} \in V$. Thus Lemma \ref{lem2b}, together with the bound $|Y| \le \ga_1 \frac{n}{\log n}$, implies that
\begin{align*}
\prod_{y \in Y} c(m_{x'y})
& \le \prod_{y \in Y} \left(1 - \frac{2c_1 \log n}{n} \sin^2 \frac{\pi m_{x'y}}{p} \right) \\
& \le \exp \left( - \frac{2c_1 \log n}{n}\sum_{y \in Y}\sin^2 \frac{\pi m_{x'y}}{p} \right) \\
& \le 1 - \frac{2c_1 \log n}{n}\sum_{y \in Y}\sin^2 \frac{\pi m_{x'y}}{p} 
+ \frac{1}{2} \left( \frac{2c_1 \log n}{n}\sum_{y \in Y}\sin^2 \frac{\pi m_{x'y}}{p} \right)^2 \\
& \le 1 - \frac{2c_1 \log n}{n} \left( 1 - \frac{c_1 \log n}{n} |Y| \right) \sum_{y \in Y}\sin^2 \frac{\pi m_{x'y}}{p} \\
& \le 1 - \frac{2c_1 \log n}{n} (1-c_1 \ga_1) \sum_{y \in Y}\sin^2 \frac{\pi m_{x'y}}{p} \\
& = 1 - \frac{2c_2 \log n}{n}\sum_{y \in Y}\sin^2 \frac{\pi m_{x'y}}{p}.
\end{align*}
Similarly, for every $y' \in Y^c$ we have
$$
\prod_{x \in X} c(m_{xy'})
\le 1 - \frac{2c_2 \log n}{n}\sum_{x \in X}\sin^2 \frac{\pi m_{xy'}}{p}.
$$
Now the above inequalities imply that \eqref{eq_3c2} is bounded above by
\begin{align*}
& \frac{1}{p^{rn}} \sum_{V \in S}
\sum_{\substack{X \subseteq [n] \\ 1 \le |X| \le \ga_1 \frac{n}{\log n}}} 
\sum_{\substack{Y \subseteq [n] \\ |Y| \le \ga_1 \frac{n}{\log n}}} 
\sum_{\substack{(g_x)_{x \in X} \in (G \setminus V)^X \\ (g_{x'})_{x' \in X^c} \in V^{X^c}}}
\sum_{\substack{(h_y)_{y \in Y} \in (G \setminus V^{\perp})^Y \\ (h_{y'})_{y' \in Y^c} \in (V^{\perp})^{Y^c}}} \\
& \prod_{x' \in X^c} \left( 1 - \frac{2c_2 \log n}{n}\sum_{y \in Y}\sin^2 \frac{\pi m_{x'y}}{p} \right) 
\prod_{y' \in Y^c} \left( 1 - \frac{2c_2 \log n}{n}\sum_{x \in X}\sin^2 \frac{\pi m_{xy'}}{p} \right) \\
\le \, & \frac{1}{p^{rn}} \sum_{V \in S}
\sum_{\substack{X \subseteq [n] \\ 1 \le |X| \le \ga_1 \frac{n}{\log n}}} 
\sum_{\substack{Y \subseteq [n] \\ |Y| \le \ga_1 \frac{n}{\log n}}} \\
& \sum_{(h_y)_{y \in Y} \in (G \setminus V^{\perp})^Y} 
\prod_{x' \in X^c} \left( \sum_{g_{x'} \in V} \left( 1 - \frac{2c_2 \log n}{n}\sum_{y \in Y}\sin^2 \frac{\pi m_{x'y}}{p} \right) \right) \\
& \sum_{(g_x)_{x \in X} \in (G \setminus V)^X}
\prod_{y' \in Y^c} \left( \sum_{h_{y'} \in V^{\perp}} \left( 1 - \frac{2c_2 \log n}{n}\sum_{x \in X}\sin^2 \frac{\pi m_{xy'}}{p} \right) \right) \\
=: \, & \frac{1}{p^{rn}} \sum_{V \in S}
\sum_{\substack{X \subseteq [n] \\ 1 \le |X| \le \ga_1 \frac{n}{\log n}}} 
\sum_{\substack{Y \subseteq [n] \\ |Y| \le \ga_1 \frac{n}{\log n}}} T(V; X,Y).
\end{align*}

For each $y \in Y$, we have $h_y \in G \setminus V^{\perp}$ and $t_{x'y}\in\fp^*$. Hence the map $V \ra \fp$ given by $g \mapsto t_{x'y}(g\cdot h_y)$ is a surjective linear map so the size of every fiber is $\frac{|V|}{p}$. Consequently,
$$
\sum_{g_{x'} \in V} \sin^2 \frac{\pi m_{x' y}}{p}
= \sum_{g_{x'} \in V} \sin^2 \frac{\pi t_{x' y}(g_{x'} \cdot h_y)}{p}
= \frac{|V|}{p} \sum_{a \in \fp} \sin^2 \frac{\pi a}{p} = \frac{|V|}{2}.
$$
Hence for every $x' \in X^c$, 
\begin{equation} \label{eq_3c4}
\begin{split}
\sum_{g_{x'} \in V} \left( 1 - \frac{2c_2 \log n}{n}\sum_{y \in Y}\sin^2 \frac{\pi m_{x'y}}{p} \right) 
& = |V| - \frac{2c_2 \log n}{n} \sum_{y \in Y} \sum_{g_{x'} \in V} \sin^2 \frac{\pi m_{x'y}}{p} \\
& = |V| \left( 1 - \frac{c_2 \log n}{n}|Y| \right).
\end{split}
\end{equation}
Similarly, for $x\in X$ and $y'\in Y^c$, the map $V^{\perp}\ra\fp$ given by $h\mapsto t_{xy'}(g_x\cdot h)$ is a surjective linear map with fibers of size $\frac{|V^{\perp}|}{p}$. Hence we obtain
\begin{equation} \label{eq_3c5}
\begin{split}
\sum_{h_{y'} \in V^{\perp}} \left( 1 - \frac{2c_2 \log n}{n}\sum_{x \in X}\sin^2 \frac{\pi m_{xy'}}{p} \right)
= |V^{\perp}| \left( 1 - \frac{c_2 \log n}{n}|X| \right).
\end{split}
\end{equation}

Assume that $n$ is sufficiently large so that $c_2(1 - \frac{\ga_1}{\log n}) > c_3 = \frac{1+c_2}{2}$. By \eqref{eq_3c4} and \eqref{eq_3c5}, 
\begin{align*}
T(V; X,Y) 
& \le |G \setminus V^{\perp}|^{|Y|} \left( |V| \left( 1 - \frac{c_2 \log n}{n}|Y| \right) \right)^{|X^c|} 
|G \setminus V|^{|X|} \left( |V^{\perp}| \left( 1 - \frac{c_2 \log n}{n}|X| \right) \right)^{|Y^c|} \\
& \le (p^r - |V^{\perp}|)^{|Y|} |V|^{n-|X|} \exp\left( -\frac{c_2 \log n}{n}|Y|(n-|X|) \right) \\
& \times (p^r - |V|)^{|X|} |V^{\perp}|^{n-|Y|} \exp\left( -\frac{c_2 \log n}{n}|X|(n-|Y|) \right) \\
& = p^{rn} \left( \frac{p^r}{|V^{\perp}|} - 1 \right)^{|Y|} \left( \frac{p^r}{|V|} - 1 \right)^{|X|}
\frac{n^{\frac{2c_2}{n}|X||Y|}}{n^{c_2(|X|+|Y|)}} \\
& \le p^{rn} (p^r-1)^{|Y|} (p^r-1)^{|X|}
\frac{n^{\frac{c_2}{n}(|X|+|Y|) \ga_1 \frac{n}{\log n}}}{n^{c_2(|X|+|Y|)}} \\
& \le p^{rn}
\frac{(p^r - 1)^{|X|+|Y|} }{n^{c_3(|X|+|Y|)}}.
\end{align*}
By the above inequality, we have
\begin{align*}
& \frac{1}{p^{rn}} \sum_{V \in S}
\sum_{\substack{X \subseteq [n] \\ 1 \le |X| \le \ga_1 \frac{n}{\log n}}} 
\sum_{\substack{Y \subseteq [n] \\ |Y| \le \ga_1 \frac{n}{\log n}}} T(V; X,Y) \\
\le \, & \sum_{V \in S}
\sum_{\substack{X \subseteq [n] \\ 1 \le |X| \le \ga_1 \frac{n}{\log n}}} 
\sum_{\substack{Y \subseteq [n] \\ |Y| \le \ga_1 \frac{n}{\log n}}} \frac{\left( p^r-1 \right)^{|X|+|Y|} }{n^{c_3(|X|+|Y|)}} \\
\le \, & \sum_{V \in S} 
\left( \sum_{x=1}^{n} \binom{n}{x}\frac{(p^r-1)^x}{n^{c_3 x}} \right)
\left( \sum_{y=0}^{n} \binom{n}{y} \frac{(p^r-1)^y}{n^{c_3 y}} \right) \\
\le \, & |S| \left( \left( 1 + \frac{p^r-1}{n^{c_3}} \right)^n - 1 \right) \left( 1 + \frac{p^r-1}{n^{c_3}} \right)^n,
\end{align*}
and the last term converges to $0$ as $n \to \infty$ since $c_3>1$.
\end{proof}

Propositions \ref{prop_case1}, \ref{prop_case2}, \ref{prop_case3a}, \ref{prop_case3b} and \ref{prop_case3c} complete the proof of Theorem \ref{thm21_main2c}. Combining Theorems \ref{thm21_main2a}, \ref{thm21_main2b} and \ref{thm21_main2c} proves Theorem \ref{thm1_main2}, and hence Theorem \ref{thm1_main}.

\bigskip
\section*{Acknowledgments}

Jungin Lee was supported by the National Research Foundation of Korea (NRF) grant funded by the Korea government (MSIT) (No. RS-2024-00334558 and No. RS-2025-02262988). 
The author thanks Jiwan Jung and Myungjun Yu for helpful discussions and comments.


\end{document}